\documentclass[10pt]{amsart}
\usepackage{amssymb}
\usepackage{amscd}
\usepackage[all]{xy}

\numberwithin{equation}{section}

\def\today{\number\day\space\ifcase\month\or   January\or February\or
   March\or April\or May\or June\or   July\or August\or September\or
   October\or November\or December\fi\   \number\year}

\theoremstyle{definition}
\newtheorem{thm}{Theorem}[section]
\newtheorem{lem}[thm]{Lemma}
\newtheorem{prp}[thm]{Proposition}
\newtheorem{dfn}[thm]{Definition}
\newtheorem{cor}[thm]{Corollary}

\newtheorem{rmk}[thm]{Remark}
\newtheorem{ntn}[thm]{Notation}
\newtheorem{exa}[thm]{Example}

\newtheorem{qst}[thm]{Question}

\newcommand{\beq}{\begin{equation}}
\newcommand{\eeq}{\end{equation}}
\newcommand{\beqr}{\begin{eqnarray*}}
\newcommand{\eeqr}{\end{eqnarray*}}
\newcommand{\bal}{\begin{align*}}
\newcommand{\eal}{\end{align*}}
\newcommand{\bei}{\begin{itemize}}
\newcommand{\eei}{\end{itemize}}

\newcommand{\af}{\alpha}
\newcommand{\bt}{\beta}
\newcommand{\gm}{\gamma}
\newcommand{\dt}{\delta}

\newcommand{\zt}{\zeta}
\newcommand{\et}{\eta}

\newcommand{\ld}{\lambda}
\newcommand{\sm}{\sigma}
\newcommand{\kp}{\kappa}
\newcommand{\ph}{\varphi}
\newcommand{\ps}{\psi}
\newcommand{\rh}{\rho}

\newcommand{\ta}{\tau}

\newcommand{\Ld}{\Lambda}

\newcommand{\Q}{{\mathbb{Q}}}
\newcommand{\Z}{{\mathbb{Z}}}

\newcommand{\C}{{\mathbb{C}}}
\newcommand{\N}{{\mathbb{Z}}_{> 0}}

\newcommand{\diag}{{\mathrm{diag}}}
\newcommand{\supp}{{\mathrm{supp}}}

\newcommand{\invlim}{\varprojlim}

\newcommand{\andeqn}{\,\,\,\,\,\, {\mbox{and}} \,\,\,\,\,\,}

\newcommand{\Wolog}{Without loss of generality}

\newcommand{\tfae}{the following are equivalent}
\newcommand{\ifo}{if and only if}

\newcommand{\ca}{C*-algebra}

\newcommand{\hm}{homomorphism}

\newcommand{\fd}{finite dimensional}
\newcommand{\pca}{pro-\ca}

\newcommand{\pfpca}{profinite pro-\ca}
\newcommand{\pfg}{profinite group}

\newcommand{\pfcmp}{profinite completion}
\newcommand{\rfd}{residually \fd}

\newcommand{\rf}{residually finite}

\newcommand{\csn}{C*~seminorm}
\newcommand{\pcas}{pro-\ca\  structure}
\newcommand{\pfpcas}{profinite pro-\ca\  structure}
\newcommand{\rpn}{representation}

\newcommand{\ct}{continuous}
\newcommand{\cfn}{continuous function}

\newcommand{\hme}{homeomorphism}

\renewcommand{\S}{\subset}
\newcommand{\ov}{\overline}
\newcommand{\wt}{\widetilde}
\newcommand{\wh}{\widehat}
\newcommand{\SM}{\setminus}
\newcommand{\I}{\infty}

\newcommand{\CNG}{{\mathcal{N}}_G}
\newcommand{\CNGb}{{\mathcal{N}}_{\overline{G}}}

\newcommand{\Lem}[1]{Lemma~\ref{#1}}
\newcommand{\Def}[1]{Definition~\ref{#1}}

\title[Pro-C*-algebras of profinite groups]{Profinite
  pro-C*-algebras and pro-C*-algebras of
  profinite groups}

\author{Rachid El Harti}
\address{Department of Mathematics and Computer Sciences,
Faculty of Sciences and Techniques, University Hassan I,
BP 577.\  Settat, Morocco}
\email[]{relharti@gmail.com}

\author{N.~Christopher Phillips}
\address{Department of Mathematics, University  of Oregon,
       Eugene OR 97403-1222, USA.}
\email[]{ncp@darkwing.uoregon.edu}

\author{Paulo R.\  Pinto}
\address{Department of Mathematics,
Instituto Superior T\'{e}cnico,
Technical University of Lisbon,
Av.\  Rovisco Pais, 1049-001 Lisboa, Portugal}
\email[]{ppinto@math.ist.utl.pt}

\date{20~April 2012}

\thanks{Some of this material is based upon work of the second
   author supported by the
   US National Science Foundation under Grant DMS-0701076.
   The second author was also partially supported
   by the Centre de Recerca Matem\`{a}tica (Barcelona)
   through a research visit conducted during 2011.
   The third author was partially supported by
   the Funda\c{c}\~{a}o para a Ci\^{e}ncia e a Tecnologia
   through the Program POCI 2010/FEDER}

\begin{document}

\begin{abstract}
We define the profinite completion of a C*-algebra,
which is a pro-C*-algebra,
as well as the pro-C*-algebra of a profinite group.
We show that the continuous representations
of the pro-C*-algebra of a profinite group
correspond to the unitary representations of the group
which factor through a finite group.
We define natural homomorphisms from the C*-algebra of a locally
compact group and its profinite completion to the pro-C*-algebra
of the profinite completion of the group.
We give some conditions for
injectivity or surjectivity of these homomorphisms,
but an important question remains open.
\end{abstract}

\maketitle

\indent
The \ca\  $C^* (G)$ of a locally compact group~$G$
is a well known construction.
One of the main motivations for it is the correspondence
between unitary representations of~$G$
and \ca\  representations of $C^* (G).$

A profinite group is an inverse limit $\invlim_{\ld \in \Ld} G_{\ld}$
of finite groups~$G_{\ld},$
with $G_{\ld}$ having the discrete topology for each~$\ld.$
Such a group is compact,
and therefore has a conventional \ca.
However, we can also form the inverse limit of the \ca s
$C^* (G_{\ld}).$
The result is a \pca,
which we introduce and study in this paper.
Its representation theory is related to a part of
the representation theory of~$G,$
namely the unitary representations which factor through finite groups.
This is the part of
the representation theory of~$G$
which should be thought of as being compatible with the structure
of~$G$ as a profinite group.

In Section~\ref{Sec:Pfcas},
we present some relevant basic facts about \pca s.
The \pca\  of a \pfg\  will be an example of a \pfpca,
so we develop the theory of these,
and of the \pfcmp\  of a \ca.
Generally this completion is far too large to be of any interest,
but it seems to be useful in connection with the \ca s of \pfg s.

Section~\ref{Sec:Pfg} contains the definition
and basic properties of the \pca\  of a \pfg.
We give here the relation between the \rpn s of the group
and its \pca.

In Section~\ref{Sec:Rel},
for a locally compact group~$G,$
we use a \hm\  from $C^* (G)$
to the \pca\  of the \pfcmp\  ${\ov{G}}$ of~$G$
to study the relation between $C^* (G)$ and its \pfcmp\  on the
one hand,
and the \pca\  of ${\ov{G}}$ on the other hand.
One can't expect a tight relationship in general.
We do prove that the map from $C^* (G)$ is injective
for a countable amenable \rf\  group
(Theorem~\ref{T:RFAmen}).
However, there are \rf\  groups whose \ca s are not \rfd;
for these, the map from $C^* (G)$ is never injective.
We do not know what happens when the
group is \rf\  but not amenable and its
C*-algebra is \rfd.
The map from the \pfcmp\  is always surjective
(Corollary~\ref{P:phBarSj}).
For a compact group,
it is an isomorphism \ifo\  the group is profinite
(Proposition~\ref{P:CompProf}),
but can be an isomorphism for noncompact groups,
sometimes for trivial reasons
(Example~\ref{E-DirSumZ2} and Example~\ref{E-InfAltGp}).

\section{Profinite pro-C*-algebras}\label{Sec:Pfcas}

\indent
In this section, we define \pca s,
\pfpca s,
and the \pfcmp\  of a (\rfd) \ca.
As will become clear,
the \pfcmp\  is usually very large.
(See Example~\ref{E:CXPfCmp}.)
In the next section,
we will see situations in which
either the \pfcmp\  is not quite so large,
or there is a completion which is smaller and more appropriate
for the situation.

We approximately follow the definition of a \pca\  of~1.2 of~\cite{Vc}.
The terminology in the literature is inconsistent;
in~\cite{Ph1} and elsewhere,
the term ``\pca'' is used for what we call here
the completion of a \pca.

\begin{dfn}\label{D:CSN}
Let $A$ be a complex *-algebra.
A {\emph{C*~seminorm}} on~$A$ is an algebra seminorm~$p$ on~$A$
which satisfies $p (a^* a) = p (a)^2$ for all $a \in A.$
If $p$ is a \csn,
we denote by $\ker (p)$ the set
\[
\ker (p) = \{ a \in A \colon p (a) = 0 \}.
\]
\end{dfn}

\begin{lem}\label{L-PropOfSeminm}
Let $A$ be a complex *-algebra and let $p$ be a \csn\  on~$A.$
Then:
\begin{enumerate}
\item\label{L-PropOfSeminm-1}
$\ker (p)$ is a *-ideal in~$A,$
closed if $A$ has a topology and $p$ is \ct.
\item\label{L-PropOfSeminm-2}
$p$~induces a norm on $A / \ker (p),$
with respect to which $A / \ker (p)$ satisfies all the axioms
for a \ca\  with the possible exception of completeness.
\item\label{L-PropOfSeminm-3}
If $A$ is a \ca,
then $p$ is automatically continuous;
in fact,  $p (a) \leq \| a \|$ for all $a \in A.$
Moreover, $A / \ker (p)$ is already complete.
\end{enumerate}
\end{lem}

\begin{proof}
The first two parts are immediate.
For the third,
set $q (a) = \max (p (a), \, \| a \| )$ for $a \in A.$
Then $q$ is a C*~norm on~$A.$
The obvious map
$A \to {\overline{A / \ker (q)}}$ is algebraically a *-\hm,
hence contractive.
Therefore $p (a) \leq \| a \|$ for all $a \in A.$
Completeness of $A / \ker (p)$
now follows from the fact that
it is the range of a \hm\  of \ca s.
\end{proof}

Lemma~\ref{L-PropOfSeminm}(\ref{L-PropOfSeminm-3})
should be compared with Corollary~5.4 of~\cite{Schm}.
However, $A$~might be smaller than
the algebra considered in~\cite{Schm}.

\begin{ntn}\label{N-CSN}
Let $A$ be a complex *-algebra.
If $p$ is a \csn\  on~$A,$
we denote by $A_p$ or ${\overline{A / \ker (p)}}$
the completion of $A / \ker (p).$
If we have a family $(p_{\ld})_{\ld \in \Ld}$
of \csn s on~$A,$
we abbreviate $A_{p_{\ld}}$ to~$A_{\ld}.$
\end{ntn}

\begin{dfn}\label{D:PCA}
A {\emph{pro-C*-algebra}} is a pair
$\big( A, (p_{\ld})_{\ld \in \Ld} \big)$
consisting of a \ca\  $A$ and a family $(p_{\ld})_{\ld \in \Ld}$
of \csn s on~$A,$
indexed by a directed set~$\Ld,$
such that $\ld \leq \mu$ implies $p_{\ld} \leq p_{\mu}.$
We also refer to $(p_{\ld})_{\ld \in \Ld}$ as a
{\emph{pro-C*-algebra structure on~$A.$}}
\end{dfn}

We have omitted two of the three conditions in~1.2 of~\cite{Vc},
because they are not satisfied in many of our examples.
We give these conditions
in the next definition.

\begin{dfn}\label{D:Full}
Let $A$ be a \ca,
and let $(p_{\ld})_{\ld \in \Ld}$
be a pro-C*-algebra structure on~$A.$
\begin{enumerate}
\item\label{D-Full-1}
We say that $(p_{\ld})_{\ld \in \Ld}$
is {\emph{full}} if the following condition is satisfied.
Whenever $(a_{\ld})_{\ld \in \Ld}$ is a family of
elements $a_{\ld} \in A_{\ld}$ which is consistent
in the sense that for $\ld \geq \mu,$ the image of $a_{\ld}$
in $A_{\mu}$ is $a_{\mu},$
and such that we also have $\| a_{\ld} \| \leq 1$ for all $\ld \in \Ld,$
then there exists $a \in A$ such that $\| a \| \leq 1$ and such that
$a + \ker ( p_{\ld} ) = a_{\ld}$ for all $\ld \in \Ld.$
\item\label{D-Faithful}
We say that $(p_{\ld})_{\ld \in \Ld}$
is {\emph{faithful}} if
$\| a \| = \sup_{\ld \in \Ld} p_{\ld} (a)$ for all $a \in A.$
\end{enumerate}
\end{dfn}

The condition of Definition~\ref{D:Full}(\ref{D-Full-1})
asserts that the inverse limit of the
closed unit balls of the quotients $A_{\ld}$
is the closed unit ball of~$A.$

Lemma~\ref{L-PropOfSeminm}(\ref{L-PropOfSeminm-3})
implies that one inequality
in Definition~\ref{D:Full}(\ref{D-Faithful})
is automatic:
we always have $\| a \| \geq \sup_{\ld \in \Ld} p_{\ld} (a).$

\begin{rmk}\label{R:Topology}
Let $A$ be a \ca, and let $(p_{\ld})_{\ld \in \Ld}$ be a
pro-C*-algebra structure on~$A.$
Then $(p_{\ld})_{\ld \in \Ld}$ defines a topology on~$A,$
in which a net $(a_{\af})_{\af \in I}$ in~$A$ converges to $a \in A$
\ifo\  $p_{\ld} (a_{\af} - a) \to 0$ for all $\ld \in \Ld.$
When $(p_{\ld})_{\ld \in \Ld}$ is understood,
we write ${\ov{A}}$ for the completion of~$A.$
This algebra is an inverse limit of \ca s
(namely, the \ca s $A / \ker (p_{\ld})$),
or a \pca\  in the sense of Definition~1.1 of~\cite{Ph1}.

The \pcas\   $(p_{\ld})_{\ld \in \Ld}$
is faithful \ifo\  the map $A \to {\ov{A}}$ is injective,
equivalently,
\ifo\  $\bigcap_{\ld \in \Ld} \ker (p_{\ld}) = \{ 0 \}.$
The \pcas\   $(p_{\ld})_{\ld \in \Ld}$
is full \ifo\  the map $A \to {\ov{A}}$ has range equal
to the \ca\  of bounded elements of ${\ov{A}}$
in the sense of Definition~1.10 of~\cite{Ph1}.
\end{rmk}

See~\cite{Ph1} and the references there for more on
the theory of inverse limits of \ca s.

\begin{dfn}\label{D:EqStruct}
Let $A$ be a \ca.
Two pro-C*-algebra structures on~$A$ are said to be
{\emph{equivalent}} if the topologies they define,
as in Remark~\ref{R:Topology},
are equal.
\end{dfn}

\begin{dfn}\label{D-PfcmpFth}
A \ca~$A$ is called {\emph{residually finite dimensional}}
if it has a faithful family of \fd\  representations.
\end{dfn}

\begin{dfn}\label{D:Pfpca}
A \pca\  $\big( A, (p_{\ld})_{\ld \in \Ld} \big)$
is {\emph{profinite}}
if $A / \ker (p_{\ld})$ is \fd\  for all $\ld \in \Ld.$
If $A$ is any \ca,
we define its {\emph{profinite pro-C*-algebra structure}}
to be the collection of all \csn s $p$ on~$A$
such that $A / \ker (p)$ is \fd\  %
(justification in Lemma~\ref{R:PfcmpIsPca} below),
and we define the {\emph{profinite completion}} of $A$ to be
$A$ equipped with this \pcas.
\end{dfn}

The \csn s in the second part of \Def{D:Pfpca} are \ct,
by Lemma~\ref{L-PropOfSeminm}(\ref{L-PropOfSeminm-3}).

\begin{lem}\label{R:PfcmpIsPca}
Let $A$ be a \ca.
Then:
\begin{enumerate}
\item\label{R:PfcmpIsPca-1}
The collection of all \csn s $p$ on~$A$
such that $A / \ker (p)$ is \fd\  is a \pca\  structure.
\item\label{R:PfcmpIsPca-2}
The profinite pro-C*-algebra structure of
Definition~\ref{D:Pfpca} is profinite.
\item\label{R:PfcmpIsPca-3}
The profinite pro-C*-algebra structure of
Definition~\ref{D:Pfpca} is faithful \ifo\  $A$ is \rfd.
\end{enumerate}
\end{lem}

\begin{proof}
All parts of the lemma are easy.
\end{proof}

We will not use this fact,
but we point out that Theorem~6.1 of~\cite{Schm}
shows that a \pca\  $\big( A, (p_{\ld})_{\ld \in \Ld} \big)$
is profinite \ifo\  ${\overline{A}}$ is semireflexive as a
topological vector space,
that is (see Sections 5.3 and~5.4 in Chapter~4 of~\cite{Schf})
the map from ${\overline{A}}$ to its strong second dual
(as a topological vector space) is bijective
(but not necessarily a homeomorphism).

Commutative \ca s are \rfd,
as is $C_0 (X, M_n)$ for any locally compact Hausdorff space~$X$
and any $n \in \N.$
Theorem~7 of~\cite{Ch} shows that the full \ca\  of the free
group on two generators is \rfd.

\begin{prp}\label{P:PFIsRFD}
Let $\big( A, (p_{\ld})_{\ld \in \Ld} \big)$ be a profinite \pca.
Set $I = \bigcap_{\ld \in \Ld} \ker (p_{\ld}).$
Then $A / I$ is \rfd.
\end{prp}

\begin{proof}
Let $a \in A \SM I.$
It suffices to find a \fd\  representation $\pi$ of $A$
such that $\pi (a) \neq 0.$
Choose $\ld$ such that $p_{\ld} (a) \neq 0,$
let $\sm$ be an injective \fd\  representation of the
\fd\  \ca\  $A / \ker (p_{\ld}),$
and take $\pi$ to be the composition of $\sm$
with the quotient map $A \to A / \ker (p_{\ld}).$
\end{proof}

\begin{cor}\label{C:PFC-RFD}
Let $A$ be a \ca,
and let $(p_{\ld})_{\ld \in \Ld}$ be the collection of
C*~seminorms of its \pfcmp.
Set $I = \bigcap_{\ld \in \Ld} \ker (p_{\ld}).$
Then $A / I$ is \rfd.
\end{cor}

\begin{proof}
The \pfcmp\  is profinite.
\end{proof}

The profinite completion of~$A$
is universal for \fd\  representations of~$A.$

\begin{prp}\label{P:PfcmpUniv}
Let $A$ be a \ca,
and let $\pi \colon A \to L (H)$ be a representation of $A$
on a Hilbert space~$H.$
Then \tfae:
\begin{enumerate}
\item\label{P:PfcmpUniv-1}
$\pi$ is \ct\  in the \pfpcas\  of~$A.$
\item\label{P:PfcmpUniv-2}
$\pi (A)$ is \fd.
\item\label{P:PfcmpUniv-3}
There is a finite set $F$ of \fd\  representations
of $A$ such that $\pi$ is a direct sum of copies of
representations in~$F.$
\end{enumerate}
Moreover, the restriction map $\sm \mapsto \sm |_A,$
from \ct\   representations of~${\ov{A}}$
to representations of~$A$ which are \ct\  %
in the \pfpcas,
is bijective.
\end{prp}

In particular,
the irreducible \ct\   representations of~${\ov{A}}$
are exactly the irreducible \fd\   representations of~$A.$

\begin{proof}[Proof of Proposition~\ref{P:PfcmpUniv}]
Let $(p_{\ld})_{\ld \in \Ld}$ be the \pfpcas\  of~$A.$

We prove the first part.
That (\ref{P:PfcmpUniv-1}) implies~(\ref{P:PfcmpUniv-2})
follows from the fact that any \ct\  \hm\  from
$A$ with any \pcas\  must factor through the quotient
by the kernel of one of the seminorms in the \pcas.
To see that (\ref{P:PfcmpUniv-2}) implies~(\ref{P:PfcmpUniv-3}),
we observe that a \fd\  \ca\  has only finitely
many unitary equivalence classes of irreducible representations,
that they are all \fd,
and that every representation is a direct sum of
irreducible representations.
That (\ref{P:PfcmpUniv-3}) implies~(\ref{P:PfcmpUniv-1})
is obvious.

The second part is clear from the following observation.
Let $\pi$ be a \rpn\  of~$A$ which is \ct\  for
the \pcas\  $(p_{\ld})_{\ld \in \Ld}.$
Then $\pi$ extends uniquely to the completion ${\ov{A}}.$
\end{proof}

\begin{prp}\label{P-IdPfCstComp}
Let $A$ be a \ca.
Equip $A$ with the \pfpcas,
and let ${\ov{A}}$ be the corresponding completion.
Let $R$ be a set consisting of one representative
$\pi \colon A \to L (H_{\pi})$
of each unitary equivalence class of \fd\  irreducible
representations of~$A.$
Give $\prod_{\pi \in R} L (H_{\pi})$ the product topology.
Then there is a unique isomorphism
$\af \colon {\ov{A}} \to \prod_{\pi \in R} L (H_{\pi})$
of topological algebras such that, for $a \in A$
with image ${\ov{a}} \in {\ov{A}},$
we have $\af ({\ov{a}}) = ( \pi (a) )_{\pi \in R}.$
\end{prp}

\begin{proof}
Let
$\af_0 \colon A \to \prod_{\pi \in R} L (H_{\pi})$
be the \hm\  given by $\af_0 (a) = ( \pi (a) )_{\pi \in R}$
for $a \in A.$
By Proposition~\ref{P:PfcmpUniv},
a net $(a_{\ld})_{\ld \in \Ld}$ in~$A$ converges to~$a \in A$
in the topology from the \pfpcas\  %
\ifo\  $\af_0 (a_{\ld}) \to \af_0 (a).$
Therefore $\af_0$ induces a unique
\ct\  \hm\  %
$\af \colon {\ov{A}} \to \prod_{\pi \in R} L (H_{\pi}).$
Moreover, if $\af$ is injective, it will follow that
$\af$ is a homeomorphism onto its image.

It therefore remains to prove that $\af$ is bijective.
We first consider injectivity.
The kernel of the map $A \to {\ov{A}}$ is the intersection
of the kernels of all \hm s to \fd\  \ca s,
and $\ker (\af_0)$ is the intersection
of the kernels of all \fd\  irreducible
representations of~$A.$
These are clearly equal,
and injectivity follows.

Since ${\overline{A}}$ is complete,
surjectivity will follow from density of $\af_0 (A)$
in $\prod_{\pi \in R} L (H_{\pi}).$
For this, it is enough to prove that if $F \subset R$ is finite,
$\pi_0 \in R \setminus F,$
and $c \in L (H_{\pi}),$
then there exists $b \in A$ such that $\pi_0 (b) = c$
and $\pi (b) = 0$ for all $\pi \in F.$
Set $I = \ker (\pi_0)$
and let $(e_{\ld})_{\ld \in \Ld}$ be an approximate identity for~$I.$
Let $\pi \in F.$
Since $\pi$ and $\pi_0$ are distinct, \fd, and irreducible,
we have $\ker (\pi_0) \not\subset \ker (\pi).$
Therefore $\pi (I) \neq 0,$ whence $\pi (I) = L (H_{\pi}).$
So $( \pi (e_{\ld}) )_{\ld \in \Ld}$
is an approximate identity for~$L (H_{\pi}).$
Thus $\pi (e_{\ld}) \to 1$ in norm.
Since $F$ is finite,
there is $\ld_0 \in \Ld$ such that
$\| 1 - \pi (e_{\ld_0}) \| < \tfrac{1}{2}$ for all $\pi \in F.$
Let $f \colon [0, 1] \to [0, 1]$ be a \cfn\  %
such that $f (0) = 0$
and $f (t) = 1$ for all $t \in \big[ \tfrac{1}{2}, 1 \big].$
Then $f ( \pi (e_{\ld_0}) ) = 1$ for all $\pi \in F.$
Set $a = 1 - f (e_{\ld_0}),$
getting $\pi_0 (a) = 1$ and $\pi (a) = 0$ for all $\pi \in F.$
Now choose $b_0 \in A$ such that $\pi (b_0) = c,$
and set $b = a b_0.$
This is the required element.
\end{proof}

\begin{exa}\label{E:CXPfCmp}
Let $X$ be a locally compact Hausdorff space.
Then $C_0 (X)$ is \rfd.
It follows from Proposition~\ref{P-IdPfCstComp}
that the topology determined by the
\pfpcas\  %
on $C_0 (X)$ is the topology of pointwise convergence,
and the completion ${\ov{C_0 (X)}}$ consists
of {\emph{all}} functions
(not even necessarily \ct\  or bounded) from $X$ to~$\C.$

In particular,
the \pfpcas\  on $C_0 (X)$ is usually not
full (Definition~\ref{D:Full}(\ref{D-Full-1})),
not even if $X$ is compact.
\end{exa}

Since we make extensive use of the multiplier algebra
$M (A)$ of a \ca~$A,$
we summarize some of its properties for convenient reference.

\begin{thm}\label{T:MultAlg}
Let $A$ be a \ca.
\begin{enumerate}
\item\label{T:MultAlg-Surj}
Let $B$ be a \ca\  and let $\ph \colon A \to B$
be a surjective \hm.
Then there exists a unique \hm\  ${\wt{\ph}} \colon M (A) \to M (B)$
such that ${\wt{\ph}} (a) = \ph (a)$ for all $a \in A.$
Moreover, ${\wt{\ph}}$ is unital.
\item\label{T:MultAlg-Rep}
Let $H$ be a Hilbert space,
and let $\pi \colon A \to L (H)$ be a nondegenerate representation.
Then there exists a unique representation
${\wt{\pi}} \colon M (A) \to L (H)$
whose restriction to~$A$ is~$\pi.$
Moreover, ${\wt{\pi}}$ is unital,
and if $\pi$ is injective then so is~${\wt{\pi}}.$
\item\label{T:MultAlg-Rfd}
If $A$ is \rfd, then so is $M (A).$
\end{enumerate}
\end{thm}

\begin{proof}
Part~(\ref{T:MultAlg-Surj}) follows from Theorem 3.1.8 of~\cite{Mr},
with $A / \ker (\ph) \cong B$ in place of~$I$ and $M (A) / \ker (\ph)$
in place of~$A.$

When $\pi$ is injective, existence
and injectivity in part~(\ref{T:MultAlg-Rep})
follow from Proposition 3.12.3 of~\cite{Pd1}.
We get the general case from this case by combining
this case with part~(\ref{T:MultAlg-Surj}).
Uniqueness follows from the fact that a representation is uniquely
determined by its restriction to any ideal for which
the restriction is nondegenerate.
It is easy to see that ${\wt{\pi}}$ is unital.

To prove part~(\ref{T:MultAlg-Rfd}),
in part~(\ref{T:MultAlg-Rep}) choose a
faithful family $(\pi_{\ld})_{\ld \in \Ld}$ of \fd\  representations
of~$A.$
Set $\pi = \bigoplus_{\ld \in \Ld} \pi_{\ld}.$
Then ${\wt{\pi}}$ is faithful
by part~(\ref{T:MultAlg-Rep}).
Uniqueness in part~(\ref{T:MultAlg-Rep})
implies that
${\wt{\pi}} = \bigoplus_{\ld \in \Ld} {\wt{\pi}}_{\ld}.$
So $M (A)$ is \rfd.
\end{proof}

\begin{rmk}\label{R:1-16A}
Let $\big( A, (p_{\ld})_{\ld \in \Ld} \big)$ be a \pca.
Then there is an induced \pcas\  on $M (A),$
given by the family $(q_{\ld})_{\ld \in \Ld}$ of \csn s
defined as follows.
For $\ld \in \Ld,$
let $\kp_{\ld} \colon A \to A / \ker (p_{\ld})$ be the quotient map.
Theorem~\ref{T:MultAlg}(\ref{T:MultAlg-Surj})
provides a \hm\  %
${\wt{\kp}}_{\ld} \colon M (A) \to M ( A / \ker (p_{\ld}) ).$
Then set $q_{\ld} (a) = \big\| {\wt{\kp}}_{\ld} (a) \big\|.$

We really get a \pcas\  this way, since if $\ld_1 \leq \ld_2$
then the \hm\  %
$\kp_{\ld_1, \ld_2} \colon
    A / \ker (p_{\ld_2}) \to A / \ker (p_{\ld_1})$
is surjective, so extends to a \hm\  %
${\wt{\kp}}_{\ld_1, \ld_2} \colon
    M ( A / \ker (p_{\ld_2}) ) \to M ( A / \ker (p_{\ld_1}) ),$
and
${\wt{\kp}}_{\ld_1, \ld_2} \circ {\wt{\kp}}_{\ld_2}
   = {\wt{\kp}}_{\ld_1}$
by uniqueness in Theorem~\ref{T:MultAlg}(\ref{T:MultAlg-Surj}).
\end{rmk}

\begin{prp}\label{P:MultRFD}
Let $A$ be a \ca.
Then the identity map of $A$ extends uniquely to a \ct\  \hm\  %
from the multiplier algebra $M (A)$ to the completion
${\ov{A}}$ of $A$ in the topology from the \pfcmp\  \pca\  structure
on~$A.$
Moreover, the resulting \hm\  is injective \ifo\  $A$ is \rfd.
\end{prp}

\begin{proof}
Let $(p_{\ld})_{\ld \in \Ld}$ be the collection of \csn s
defining the \pca\  structure of the \pfcmp\  of~$A.$
For $\ld, \mu \in \Ld$ with $\ld \leq \mu,$
further let
$\kp_{\ld} \colon A \to A / \ker (p_{\ld})$ and
$\kp_{\ld, \mu} \colon A / \ker (p_{\mu}) \to A / \ker (p_{\ld})$
be the quotient maps.

It follows from Theorem~\ref{T:MultAlg}(\ref{T:MultAlg-Surj})
that there is a unique family of \hm s
$\ph_{\ld} \colon M (A) \to A / \ker (p_{\ld})$
such that $\ph_{\ld}$
extends the quotient map $\kp_{\ld} \colon A \to A / \ker (p_{\ld}).$
Uniqueness further implies that
whenever $\ld, \mu \in \Ld$ with $\ld \leq \mu,$
we have $\kp_{\ld, \mu} \circ \ph_{\mu} = \ph_{\ld}.$
It follows that there is
a unique \hm\  $\ph \colon M (A) \to {\ov{A}}$
such that $\kp_{\ld} \circ \ph = \ph_{\ld}$ for all $\ld \in \Ld.$
It is clear that $\ph$ is the unique \ct\  \hm\  from $M (A)$
to ${\ov{A}}$ which extends the identity map of~$A.$

Now assume $A$ is \rfd;
we prove that $\ph$ is injective.
Let $x \in M (A)$ be nonzero.
Choose $a \in A$ such that $x a \neq 0.$
The element $x a$ is in~$A,$
so there is $\ld \in \Ld$ such that $p_{\ld} (x a) \neq 0.$
Then $\ph_{\ld} (x) \kp_{\ld} (a) = \kp_{\ld} (x a) \neq 0,$
so $\ph_{\ld} (x) \neq 0.$
Thus $\ph (x) \neq 0.$

Finally, assume that $\ph$ is injective;
we prove that $A$ is \rfd.
By construction,
${\ov{A}}$ has a faithful family of \ct\  \fd\  representations
on Hilbert spaces.
Therefore any \ca\  contained in ${\ov{A}}$ is \rfd.
Since $\ph$ is injective,
the conclusion follows.
\end{proof}

\begin{prp}\label{P:FunctPFC}
Let $A$ and $B$ be \ca s,
and let $\ph \colon A \to B$ be a \hm.
Then $\ph$ is \ct\  with respect to the topologies
on $A$ and $B$ coming from their \pfcmp s.
\end{prp}

\begin{proof}
This follows immediately from the fact that
if $\pi$ is a \fd\  representation of~$B,$
then $\pi \circ \ph$ is a \fd\  representation of~$A.$
\end{proof}

As is clear from Example~\ref{E:CXPfCmp},
the \pfcmp\  of a \ca\  is in general extremely large.
Here is a situation in which it is not so bad.
We will apply the result to the \ca s of compact groups.

\begin{prp}\label{P:1-19A}
Let $S$ be a set, and for $s \in S$ let $n (s) \in \N.$
Set $A = \bigoplus_{s \in S} M_{n (s)},$
and give it the \pfpcas\  (\Def{D:Pfpca}).
Further give $M (A)$ the \pcas\  coming from Remark~\ref{R:1-16A}.
Then:
\begin{enumerate}
\item\label{P:1-19A-1}
We have
\[
M (A) = \left\{ a \in \prod_{s \in S} M_{n (s)} \colon
          \sup_{s \in S} \| a_{s} \| < \I \right\}.
\]
\item\label{P:1-19A-2}
The \pca\  topologies (\Def{R:Topology}) on $A$ and $M (A)$
both come from the obvious identifications of $A$ and $M (A)$
with subsets of $\prod_{s \in S} M_{n (s)}$ by restricting
the product topology.
\item\label{P:1-19A-3}
Both ${\ov{A}}$ and ${\ov{M (A)}}$
are equal to $\prod_{s \in S} M_{n (s)}.$
\item\label{P:1-19A-4}
The algebras of bounded elements
(in the sense of Definition~1.10 of~\cite{Ph1})
in both ${\ov{A}}$ and ${\ov{M (A)}}$
are equal to $M (A).$
\item\label{P:1-19A-5}
The \pcas\  on~$A$ is faithful
(Definition~\ref{D:Full}(\ref{D-Faithful})).
\item\label{P:1-19A-6}
The \pcas\  on $M (A)$ is faithful
and full (Definition~\ref{D:Full}(\ref{D-Full-1})).
\end{enumerate}
\end{prp}

\begin{proof}
Part~(\ref{P:1-19A-1}) is clear.

Part~(\ref{P:1-19A-2}) for~$A$
follows from Proposition~\ref{P-IdPfCstComp}.
For $M (A),$
by Proposition~\ref{P-IdPfCstComp} and the definition of
the \pcas\  on $M (A),$
the \csn s in this \pcas\  are exactly the seminorms
$p_F \big( (a_{s})_{s \in S} \big) = \sup_{s \in F} \| a_s \|$
for finite sets $F \subset S.$
These define the product topology on $\prod_{s \in S} M_{n (s)},$
proving the statement about $M (A).$

The remaining parts of the proposition are now clear.
\end{proof}

\section{Pro-C*-algebras of profinite groups}\label{Sec:Pfg}

\indent
In this section,
we define the \pca\  of a \pfg~$G,$
and connect it to the \rpn\  theory of~$G.$
The construction makes sense for any locally compact group~$G,$
although in general it reflects the behavior of the \pfcmp\  of~$G$
rather than of~$G.$
We give the construction in this generality,
since we use it in Section~\ref{Sec:Rel}
to compare the \pfpcas\  on $C^* (G)$
with the \pca\  of the \pfcmp\  of~$G.$

We recall the following definition.
(See the beginning of Section~2.1 of~\cite{RZ}.)

\begin{dfn}\label{D:PFG}
A topological group~$G$ is said to be {\emph{profinite}}
if $G$ is topologically isomorphic to an inverse limit
$\varprojlim_{\ld} G_{\ld}$ of an inverse system of finite
groups $G_{\ld},$ in which $G_{\ld}$ is given the discrete
topology for each $\ld \in \Ld.$
\end{dfn}

In particular, a profinite group is compact,
Hausdorff,
and totally disconnected.
Conversely, every compact
Hausdorff
totally disconnected group is profinite.
See Theorem 2.1.3 of~\cite{RZ}.
See~\cite{RZ} for much more on profinite groups.

The following notation is based on Section~3.1 of~\cite{RZ}.

\begin{ntn}\label{N:NG}
For any group $G$ and any subgroup~$H,$
we denote by $[G : H]$ the index of $H$ in~$G.$
For any topological group~$G$
(assuming the discrete topology if no other topology is obvious
or specified),
we let $\CNG$ be the set of closed normal subgroups $N \S G$
such that the index $[G : N]$ is finite.
We order $\CNG$ by reverse inclusion.
\end{ntn}

\begin{rmk}\label{R:NGDir}
Let $G$ be a topological group.
If $M, N \in \CNG$ then also $M \cap N \in \CNG.$
(In fact,
$[G : M \cap N] \leq [G : M] \cdot [G : N].$)
Therefore $\CNG$ is a directed set.
\end{rmk}

The following definition is at the beginning of Section~3.2
of~\cite{RZ}.

\begin{dfn}\label{D-PfCompGp}
Let $G$ be a topological group.
We define its {\emph{profinite completion}} ${\ov{G}}$
to be the inverse limit of the quotients $G / N$ as
$N$ runs through~$\CNG.$
\end{dfn}

The \pfcmp\  ${\ov{G}}$ is obviously a profinite group.

\begin{ntn}\label{N:NBar}
Let $G$ be a topological group.
There is an obvious \ct\  \hm\  from $G$ to~${\ov{G}},$
which we denote by $\gm_G.$
For $N \in \CNG,$
we write ${\ov{N}}$ for the closure ${\ov{\gm_G (N)}}$
of the image of $N$ in~${\ov{G}}.$
\end{ntn}

\begin{lem}\label{R:NCorr}
Let $G$ be a topological group.
Then $N \mapsto {\ov{N}}$
defines a bijection from $\CNG$ to~$\CNGb.$
Moreover,
for every $N \in \CNG,$ the map $\gm_G$ induces an isomorphism
$G / N \to {\ov{G}} \big/ {\ov{N}}.$
\end{lem}

\begin{proof}
See parts (a), (b), and~(d) of Proposition 3.2.2 of~\cite{RZ}.
\end{proof}

We recall that a topological group~$G$
is called {\emph{residually finite}}
if the intersection
of the closed normal subgroups of finite index in~$G$
is $\{ 1 \}.$
Clearly, then,
the map $G \to {\ov{G}}$ is injective \ifo\  $G$
is residually finite.

\begin{rmk}\label{R:UGMG}
We recall that if $G$ is a locally compact group,
then there is a
standard \ct\  unital \hm\  from the measure algebra $M (G)$
to the multiplier algebra $M (C^* (G))$
which extends the \hm\  $L^1 (G) \to C^* (G)$
coming from the definition of $C^* (G)$
as the universal enveloping \ca\  of $L^1 (G).$
(See 7.1.5 of~\cite{Pd1}.)
For each $g \in G,$ there is then a unitary $u_g \in M (C^* (G))$
obtained as the image of the point mass measure at~$g,$
regarded as an element of $M (G).$

We further recall the correspondence between unitary representations
of~$G$ and nondegenerate representations of $C^* (G).$
(See Proposition 7.1.4 of~\cite{Pd1} and its proof.)
Let $\mu$ be a left Haar measure on~$G.$
If $v \mapsto v_g \colon G \to U (H)$ is a unitary representation of~$G$
on a Hilbert space~$H,$
then the corresponding representation $\pi \colon C^* (G) \to L (H)$
is defined on $L^1 (G)$ by the formula, for $\xi, \et \in H,$
\[
\langle \pi (f) \xi, \et \rangle
  = \int_G f (g) \langle v_g \xi, \et \rangle \, d \mu (g).
\]
Given a nondegenerate representation~$\pi,$
we recover $v$ as follows.
First let ${\wt{\pi}}$ be the unique
extension to a unital \hm\  from $M (C^* (G))$ to $L (H),$
as in Theorem~\ref{T:MultAlg}(\ref{T:MultAlg-Rep}).
Then, with $u_g$ as in the previous paragraph,
we have $v_g = {\wt{\pi}} (u_g).$
\end{rmk}

\begin{lem}\label{L:ToFiniteGp}
Let $G$ be a locally compact group, let $F$ be a finite group,
and let $\rh \colon G \to F$ be a surjective continuous \hm.
Then there is a
unique surjective \hm\  $\pi \colon C^* (G) \to C^* (F)$
whose extension to a \hm\  ${\wt{\pi}} \colon M (C^* (G)) \to C^* (F)$
satisfies, in the notation of Remark~\ref{R:UGMG},
the relation ${\wt{\pi}} (u_g) = u_{\rh (g)}$ for all $g \in G.$
For $f \in C_{\mathrm{c}} (G),$
and with $\mu$ being a left Haar measure on~$G,$
it is given by the formula
\begin{equation}\label{Eq:FormulaForPi}
\pi (f)
 = \sum_{x \in F} \left( \int_{\rh^{-1} (x)} f \, d \mu \right) u_x.
\end{equation}
\end{lem}

\begin{proof}
We may assume that $C^* (F)$ is a unital subalgebra of
$L (H)$ for some \fd\  Hilbert space~$H.$
Then everything follows from Remark~\ref{R:UGMG}
except $\pi ( C^* (G) ) = C^* (F)$
and~(\ref{Eq:FormulaForPi}).
For~(\ref{Eq:FormulaForPi}),
let $\ph (f)$ denote the right hand side;
then a calculation shows that for all $\xi, \et \in l^2 (F),$
we have
$\langle \ph (f) \xi, \, \et \rangle
   = \langle \pi (f) \xi, \, \et \rangle.$
It is
now clear that $\pi (f) \in C^* (F).$
Therefore $\pi (a) \in C^* (F)$ for all $a \in C^* (G).$

It remains only to prove that $\pi$ is surjective.
It is enough to show that for every $x \in F,$
the unitary $u_x \in C^* (F)$ is in the range of~$\pi.$
Set $W = \rh^{-1} (x),$
which is a nonempty open and closed subset of~$G.$
Then there is $f \in C_{\mathrm{c}} (G)$ such that
$\supp (f) \S W$ and $\int_G f \, d \mu = 1.$
The formula for $\pi (f)$ implies that $\pi (f) = u_x.$
\end{proof}

\begin{lem}\label{L:Compat}
Let $G$ be a locally compact group, let $F_1$ and $F_2$
be finite groups,
and let
\[
\rh_1 \colon G \to F_1,
\,\,\,\,\,\,
\rh_2 \colon G \to F_2,
\andeqn
\gm \colon F_1 \to F_2
\]
be surjective continuous \hm s such that $\rh_2 = \gm \circ \rh_1.$
Let
\[
\pi_1 \colon C^* (G) \to C^* (F_1),
\,\,\,\,\,\,
\pi_2 \colon C^* (G) \to C^* (F_2)
\andeqn
\ph \colon C^* (F_1) \to C^* (F_2)
\]
be the \hm s of Lemma~\ref{L:ToFiniteGp},
and let ${\wt{\pi}}_1$ and ${\wt{\pi}}_2$
be the extensions of $\pi_1$ and $\pi_2$
as in Theorem~\ref{T:MultAlg}(\ref{T:MultAlg-Rep}).
Then $\pi_2 = \ph \circ \pi_1$
and ${\wt{\pi}}_2 = \ph \circ {\wt{\pi}}_1.$
\end{lem}

\begin{proof}
Clearly $\pi_2$ and $\ph \circ \pi_1$ are both surjective.
Using the fact that $M (C^* (F_1)) = C^* (F_1)$
(since $C^* (F_1)$ is unital),
we check that $\pi_2 (u_g) = (\ph \circ \pi_1) (u_g)$
for all $g \in G.$
So $\pi_2 = \ph \circ \pi_1$
by the uniqueness statement in
Lemma~\ref{L:ToFiniteGp}.
That ${\wt{\pi}}_2 = \ph \circ {\wt{\pi}}_1$ now follows from
the uniqueness statement
in Theorem~\ref{T:MultAlg}(\ref{T:MultAlg-Rep}).
\end{proof}

\begin{prp}\label{P:BddIsMult}
Let $G$ be any compact group.
Let ${\ov{C^* (G)}}$ be the completion of $C^* (G)$
in its \pfpcas.
Then the \ca\  $B$ of bounded elements of ${\ov{C^* (G)}}$
(in the sense of Definition~1.10 of~\cite{Ph1})
is unital and contains $C^* (G)$ as an ideal,
and the induced map $M (C^* (G)) \to B$ is an isomorphism of \ca s.
Moreover, $M (C^* (G))$ is \rfd.
\end{prp}

In particular, the \pfpcas\  of $M (C^* (G))$
is full in the sense of Definition~\ref{D:Full}(\ref{D-Full-1}).

\begin{proof}[Proof of Proposition~\ref{P:BddIsMult}]
Since all irreducible representations of $G$ are \fd,
there are an index set $S$ and numbers $n (s) \in \N$ such that
we can write $C^* (G)$ as a direct sum
$C^* (G) = \bigoplus_{s \in S} M_{n (s)}.$
The result now follows from Proposition~\ref{P:1-19A}.
\end{proof}

\begin{dfn}\label{D:ProCStPfGp}
Let $G$ be a locally compact group.
For each $N \in \CNG$ (Notation~\ref{N:NG}),
we define a C*~seminorm $\| \cdot \|_N$ on $C^* (G)$
by letting $\kp_N \colon C^* (G) \to C^* (G / N)$
be the \hm\  obtained from Lemma~\ref{L:ToFiniteGp} using
the quotient map,
and setting $\| a \|_N = \| \kp_N (a) \|.$
We define the {\emph{pro-C*-algebra}} of~$G$ to be $C^* (G)$
equipped with the \pca\  structure consisting of the
seminorms $\| \cdot \|_N$ as $N$ runs through~$\CNG.$
Letting ${\wt{\kp}}_N \colon M (C^* (G)) \to C^* (G / N)$
be the extension as in Theorem~\ref{T:MultAlg}(\ref{T:MultAlg-Rep}),
we further define the
{\emph{multiplier pro-C*-algebra}} of~$G$ to be $M (C^* (G))$
equipped with the \pca\  structure consisting of the
seminorms $a \mapsto \| {\wt{\kp}}_N (a) \|$
as $N$ runs through~$\CNG.$
\end{dfn}

Before proving elementary facts about these \pcas s,
we give a lemma which will be used again later.

\begin{lem}\label{L:FinRange}
Let $G$ be a profinite group.
Let $v$ be a unitary representation of~$G$
on a \fd\  Hilbert space~$H.$
Then there is $N \in \CNG$
and a unitary representation $w$ of~$G / N$ on~$H$
such that $v_g = w_{g N}$ for all $g \in G.$
\end{lem}

\begin{proof}
Choose an open set $W$ of the unitary group $U (H)$
such that $W$ contains no subgroups of $U (H)$ other than $\{ 1 \}.$
Let
\[
V = \{ g \in G \colon v_g \in W \}.
\]
Then $V$ is an open subset of~$G.$
Since $G$ is profinite,
there is $N \in \CNG$ such that $N \subset V.$
Since $W$ contains no nontrivial subgroups,
it follows that $v_g = 1$ for all $g \in N.$
Therefore $v$ induces a representation
$w$ of $G / N$ on~$H.$
\end{proof}

\begin{prp}\label{L:IsPCSStr}
Let $G$ be a locally compact group.
\begin{enumerate}
\item\label{L:IsPCSStr-1}
The collection of seminorms $\| \cdot \|_N$
of \Def{D:ProCStPfGp} is a \pcas\  on $C^* (G),$
and the collection of seminorms
$a \mapsto \| {\wt{\kp}}_N (a) \|$
of \Def{D:ProCStPfGp} is a \pcas\  on $M (C^* (G)).$
\item\label{L:IsPCSStr-2}
Suppose, in addition,
that $G$ is profinite.
Then the \pcas\  on $C^* (G)$
is faithful and is equivalent (\Def{D:EqStruct})
to the \pfpcas\  (\Def{D:Pfpca}) of $C^* (G).$
The \pcas\  on $M (C^* (G))$ is faithful and full.
\end{enumerate}
\end{prp}

\begin{proof}
We prove~(\ref{L:IsPCSStr-1}).
\Lem{L:Compat} implies that if $N_1, N_2 \in \CNG$
satisfy $N_2 \S N_1,$
then $\| a \|_{N_1} \leq \| a \|_{N_2}$ for all $a \in C^* (G).$
Moreover, $\CNG$ is directed by Remark~\ref{R:NGDir}.
Thus, we have \pca\  structures on $C^* (G)$ and $M (C^* (G)).$

Now assume that $G$ is profinite.

We first prove that the \pca\  structure on $C^* (G)$ is equivalent to
the one defining the profinite completion of $C^* (G).$
Since the \ca s of finite groups are \fd,
it suffices to prove that if $p$ is a \csn\  on $C^* (G)$
such that $C^* (G) / \ker (p)$ is \fd,
then there is closed normal subgroup~$N$ of finite index in~$G$
such that $\| \cdot \|_N \geq p.$

Represent $C^* (G) / \ker (p)$ unitally and faithfully
on a \fd\  Hilbert space~$H.$
Thus, we have a \hm\  $\pi \colon C^* (G) \to L (H)$
whose range contains~$1$
and such that $p (a) = \| \pi (a) \|$ for all $a \in C^* (G).$
Then $\pi$ comes from a unitary representation
$v \mapsto v_g$ of~$G$ on~$H.$
Let $N$ and $w \colon G / N \to L (H)$
be as in \Lem{L:FinRange}.
Let $\ps \colon C^* (G / N) \to L (H)$ be the corresponding
representation of $C^* (G / N).$
Then $\ps \circ \kp_N$ and $\pi$
are both nondegenerate \hm s from $C^* (G)$ to $L (H)$
whose extensions to \hm s $M (C^* (G)) \to L (H)$
send $u_g$ to~$v_g$ for all $g \in G.$
Therefore $\ps \circ \kp_N = \pi.$
For all $a \in A,$
we thus have $\| \kp_N (a) \| \geq \| \pi (a) \| = p (a).$
This proves the equivalence of \pca\  structures.

The rest of~(\ref{L:IsPCSStr-2})
now follows from Proposition~\ref{P:BddIsMult}.
\end{proof}

The following is the analog for profinite groups of
Proposition~\ref{P:PfcmpUniv}.

\begin{prp}\label{P:UnivFinRep}
Let $G$ be a profinite group.
Then a nondegenerate representation $\pi$ of the \pca\  of~$G$
(in the sense of \Def{D:ProCStPfGp})
is \ct\  \ifo\  the image of $G$ under the corresponding
unitary representation of~$G$ is finite.
In particular,
there is a bijective correspondence between
\ct\  nondegenerate representations of the \pca\  of~$G$
and unitary representations of~$G$ with finite range.
\end{prp}

\begin{proof}
The first part is just the statement that
$\pi$ is \ct\  \ifo\  it factors through the \ca\  of a finite
quotient of~$G.$

The second part follows from the first and from Remark~\ref{R:UGMG}.
\end{proof}

\section{The relation between the profinite completion of $C^* (G)$
  and the C*-algebra of the profinite completion of~$G$}\label{Sec:Rel}

\indent
In this section, we construct a \hm~$\ph_G$ from
the \ca\  of a locally compact group~$G$ to the completion
of the \pca\  of
its profinite completion~${\ov{G}},$
and we study its properties.
Two questions seem to be interesting:
when is $\ph_G$ injective,
and when is the extension of $\ph_G$ to the \pfcmp\  of $C^* (G)$
surjective or an isomorphism?
The first question is really about injectivity of a
\hm\  $C^* (G) \to M \big( C^* \big( {\ov{G}} \big) \big).$
It thus does not involve \pca s,
although they provide the motivation.
We give a positive answer when $G$ is residually finite
and either discrete amenable or abelian.
Residual finiteness is necessary for fairly trivial reasons,
but we show by example that it is not sufficient.
For the second question,
we prove that the extension is always surjective.
If $G$ is profinite,
then the extension is an isomorphism.
If $G$ is compact, the converse is true,
but there are noncompact groups
for which the extension is an isomorphism.

\begin{prp}\label{P:MapToCstPfc}
Let $G$ be a locally compact group,
with profinite completion ${\ov{G}},$
and let $\gm_G \colon G \to {\ov{G}}$ be the canonical map,
as in Notation~\ref{N:NBar}.
For $N \in \CNG,$
let $\rh_N \colon G \to G / N$ be the quotient map,
and let $\kp_N \colon C^* (G) \to C^* (G / N)$
be the map of \Lem{L:ToFiniteGp}.
Further, following Proposition~\ref{P:BddIsMult},
as applied to~${\ov{G}},$
let
$\sm_N \colon M \big( C^* \big( {\ov{G}} \big) \big)
     \to C^* (G / N)$
be the restriction to $M \big( C^* \big( {\ov{G}} \big) \big)$
of the map ${\ov{C^* \big( {\ov{G}} \big)}} \to C^* (G / N)$
which comes via \Def{D:ProCStPfGp}
from the expression of ${\ov{G}}$
as the inverse limit of the finite quotients $G / N.$
Then there exists a unique \hm\  %
$\ph_G \colon C^* (G) \to M \big( C^* \big( {\ov{G}} \big) \big)$
such that $\sm_N \circ \ph_G = \kp_N$ for all~$N.$
Moreover:
\begin{enumerate}
\item\label{P:MapToCstPfc-1}
$\ph_G$ extends to a \hm\  ${\wt{\ph}}_G$ from
$M (C^* (G))$ to $M \big( C^* \big( {\ov{G}} \big) \big)$
such that, in the notation of Remark~\ref{R:UGMG},
we have ${\wt{\ph}}_G (u_g) = u_{\gm_G (g)}$ for all $g \in G.$
\item\label{P:MapToCstPfc-2}
$\ph_G$ is nondegenerate,
that is,
${\ov{\ph_G (C^* (G)) \cdot C^* \big( {\ov{G}} \big)}}
 = C^* \big( {\ov{G}} \big).$
\end{enumerate}
\end{prp}

\begin{proof}
For each $M \in \CNG,$
let
$\ta_M \colon \invlim_{N \in \CNG} C^* (G / N) \to C^* (G / M)$
be the standard map from the inverse limit.
Using \Lem{L:Compat} to check compatibility of the maps,
we see that there is a unique \hm\  %
$\ps \colon C^* (G) \to \invlim_{N \in \CNG} C^* (G / N)$
such that $\ta_N \circ \ps = \kp_N$ for all~$N.$
Since $C^* (G)$ is a \ca,
the range of $\ps$ lies in the \ca\  of bounded elements
of the inverse limit.
In the notation of Proposition~\ref{P:BddIsMult},
applied to the compact group~${\ov{G}},$
this inverse limit is ${\ov{C^* \big( {\ov{G}} \big)}}.$
It follows from Proposition~\ref{P:BddIsMult}
that the bounded elements can be identified with
$M \big( C^* \big( {\ov{G}} \big) \big).$
Thus we can take $\ph_G$ to be the corestriction of $\ps$
to $M \big( C^* \big( {\ov{G}} \big) \big).$
Uniqueness of~$\ph_G$ is obvious from uniqueness of~$\ps.$

We prove~(\ref{P:MapToCstPfc-1}).
Let ${\wt{\kp}}_N$ be the extension of $\kp_N$ to
a map $M (C^* (G)) \to C^* (G / N).$
It follows from \Lem{L:ToFiniteGp}
that ${\wt{\kp}}_N (u_g) = u_{g N}$ for all $g \in G.$
The maps ${\wt{\kp}}_N$ are also compatible with the inverse system
defining ${\ov{C^* \big( {\ov{G}} \big)}},$
and therefore give a map
${\ov{\ph}}_G \colon M (C^* (G)) \to {\ov{C^* \big( {\ov{G}} \big)}}$
such that ${\ov{\ph}}_G (u_g) = u_{\gm_G (g)}$ for all $g \in G.$
For the same reasons as for $\ph_G,$
the range of ${\ov{\ph}}_G$
is contained in $M \big( C^* \big( {\ov{G}} \big) \big).$

We now prove nondegeneracy (part~(\ref{P:MapToCstPfc-2})).
Let $\mu$ be a left Haar measure on~$G.$
Let ${\ov{\mu}}$ be normalized left Haar measure on~${\ov{G}}.$

We claim that if $M, N \in \CNG$ with $M \S N,$
if $f \in C_{\mathrm{c}} (G)$ is supported in~$N,$
and if $a \in C \big( {\ov{G}} \big)$
is constant on cosets of~${\ov{N}},$
then
\[
\kp_M (f) \sm_M (a) = \left( \int_G f \, d \mu \right) \sm_N (a).
\]
To prove this,
let $S$ be a set of coset representatives for $M$ in~$G.$
Then $\gm_G (S)$ is a set of coset representatives
for ${\ov{M}}$ in~${\ov{G}},$
by Lemma~\ref{R:NCorr}.
Also, for $x \in G / M$
denote the corresponding unitary in $C^* (G / M)$ by~$u_x.$
We use the formula of \Lem{L:ToFiniteGp} at the first step
and $\supp (f) \S N$ at the second step to write
\[
\kp_M (f)
 = \sum_{g \in S} \left( \int_{g M} f \, d \mu \right) u_{g M}
 = \sum_{g \in S \cap N} \left( \int_{g M} f \, d \mu \right) u_{g M}.
\]
Using the uniqueness part of \Lem{L:ToFiniteGp},
we see that
\[
\sm_M |_{C^* \left( {\ov{G}} \right)} \colon
     C^* \big( {\ov{G}} \big) \to C^* (G / M)
\]
is also a \hm\  of the form of \Lem{L:ToFiniteGp}.
Using the formula of \Lem{L:ToFiniteGp} at the first step,
and the fact that $a$ is constant on cosets of~${\ov{N}}$
and ${\ov{M}} \S {\ov{N}}$ at the second step,
we thus get
\[
\sm_M (a)
 = \sum_{h \in S}
     \left( \int_{\gm_G (h) {\ov{M}}} a \, d {\ov{\mu}} \right) u_{h M}
 = [G : M]^{-1} \sum_{h \in S} a ( \gm_G (h)) u_{h M}.
\]
Now multiply:
\[
\kp_M (f) \sm_M (a)
  = [G : M]^{-1} \sum_{g \in S \cap N}
    \left( \int_{g M} f \, d \mu \right)
     \sum_{h \in S} a (\gm_G (h)) u_{g h M}.
\]
Since $a$ is constant on cosets of~${\ov{N}},$
we can replace $a (\gm_G (h))$ by $a (\gm_G (g h)).$
Because $g S$ is also a system of coset representatives
of ${\ov{M}}$ in~${\ov{G}},$
we have
\[
[G : M]^{-1} \sum_{h \in S} a (\gm_G (g h)) u_{g h M} = \sm_M (a).
\]
Since $f$ vanishes off~$N,$
we therefore get
\[
\kp_M (f) \sm_M (a)
 = \sum_{g \in S \cap N} \left( \int_{g M} f \, d \mu \right) \sm_M (a)
 = \left( \int_{G} f \, d \mu \right) \sm_M (a).
\]
This proves the claim.

We next claim that if $f$ and $a$ are as in the claim,
then $\ph_G (f) a =  \left( \int_{G} f \, d \mu \right) a.$
This follows from the fact that the intersection of
the kernels $\ker (\sm_M),$
as $M$ runs through all $M \in \CNG$ such that $M \S N,$
is zero.
We can now choose $f \in C_{\mathrm{c}} (G)$ with $\supp (f) \subset N$
such that $\int_G f \, d \mu = 1,$
getting $\ph_G (f) a = a.$

To complete the proof of nondegeneracy,
it is enough to prove that the set
\[
T = \big\{ a \in C \big( {\ov{G}} \big) \colon
   {\mbox{there is $N \in \CNG$ such that
   $a$ is constant on cosets of~${\ov{N}}$}} \big\}
\]
is dense in $C^* \big( {\ov{G}} \big).$
It follows from uniform continuity of elements of
$C \big( {\ov{G}} \big)$
that $T$ is dense in $C \big( {\ov{G}} \big).$
The proof is completed
by observing that $C \big( {\ov{G}} \big)$
is dense in $C^* \big( {\ov{G}} \big).$
\end{proof}

The main business of this section is the study of the maps
$\ph_G$ and ${\wt{\ph}}_G.$

The first question we consider is when $\ph_G$ is injective.
Formally, this question doesn't involve \pca s at all;
they merely provide the motivation.

\begin{prp}\label{P:Factor}
Let $G$ be a locally compact group,
with profinite completion ${\ov{G}},$
with canonical map $\gm_G \colon G \to {\ov{G}},$
and with
$\ph_G \colon C^* (G) \to M \big( C^* \big( {\ov{G}} \big) \big)$
as in Proposition~\ref{P:MapToCstPfc}.
Let $H$ be a Hilbert space,
let $g \mapsto v_g$ be a unitary representation of $G$ on~$H,$
and let $\pi \colon C^* (G) \to L (H)$ be the corresponding
nondegenerate representation of $C^* (G).$
Then \tfae:
\begin{enumerate}
\item\label{P:Factor-Gp}
The representation $g \mapsto v_g$ factors through ${\ov{G}}$:
there exists a unitary representation $x \mapsto w_x$
of ${\ov{G}}$ on $H$ such that $v_g = w_{\gm_G (g)}$ for all $g \in G.$
\item\label{P:Factor-Cst}
The representation $\pi$
factors through $M \big( C^* \big( {\ov{G}} \big) \big)$:
there exists a nondegenerate representation
$\mu \colon C^* \big( {\ov{G}} \big) \to L (H)$
whose extension ${\wt{\mu}}$ to
$M \big( C^* \big( {\ov{G}} \big) \big)$
satisfies $\pi = {\wt{\mu}} \circ \ph_G.$
\end{enumerate}
\end{prp}

\begin{proof}
Assume~(\ref{P:Factor-Gp}).
Since ${\ov{G}}$ is compact,
$w$ is a direct sum of \fd\  representations.
It therefore suffices to prove~(\ref{P:Factor-Cst}) under
the assumption that $w$ is \fd.

Apply Lemma~\ref{L:FinRange} to~$w,$
with ${\ov{G}}$ in place of~$G,$
and use Lemma~\ref{R:NCorr},
to find $N \in \CNG$
and a unitary representation $w^{(0)}$
of ${\ov{G}} / {\ov{N}}$ on~$H$
such that $w_x = w^{(0)}_{x {\ov{N}} }$ for all $x \in {\ov{G}}.$
Let $\mu \colon C^* \big( {\ov{G}} \big) \to L (H)$
and
$\mu^{(0)} \colon
    C^* \big( {\ov{G}} / {\ov{N}} \big) \to L (H)$
be the nondegenerate representations of \ca s
corresponding to $w$ and $w^{(0)}.$
Let
$\nu \colon
 C^* \big( {\ov{G}} \big)
  \to C^* \big( {\ov{G}} / {\ov{N}} \big)$
be the map of \Lem{L:ToFiniteGp}.
Thus $\mu^{(0)} \circ \nu = \mu.$
Let
\[
{\wt{\mu}} \colon
   M \big( C^* \big( {\ov{G}} \big) \big)
  \to L (H)
\andeqn
{\wt{\nu}} \colon
   M \big( C^* \big( {\ov{G}} \big) \big)
  \to C^* \big( {\ov{G}} / {\ov{N}} \big)
\]
be the extensions of $\mu$ and~$\nu.$
Following Lemma~\ref{R:NCorr},
identify ${\ov{G}} / {\ov{N}}$ with $G / N,$
and thus identify ${\wt{\nu}}$ with the map $\sm_N$
in the statement of Proposition~\ref{P:MapToCstPfc}.
Also let $\kp_N \colon C^* (G) \to C^* (G / N)$ be as there.
Then, using Proposition~\ref{P:MapToCstPfc} at the second step,
\[
{\wt{\mu}} \circ \ph_G
  = \mu^{(0)} \circ {\wt{\nu}} \circ \ph_G
  = \mu^{(0)} \circ \kp_N.
\]
It follows from
Proposition~\ref{P:MapToCstPfc}(\ref{P:MapToCstPfc-2})
that ${\wt{\mu}} \circ \ph_G$ is nondegenerate.
Let
\[
{\wt{\ph}}_G \colon
 M (C^* (G)) \to M \big( C^* \big( {\ov{G}} \big) \big)
\andeqn
{\wt{\kp}}_N \colon
 M (C^* (G)) \to C^* (G / N)
\]
be the extensions of $\ph_G$ and $\kp_N$
(the first coming from
Proposition~\ref{P:MapToCstPfc}(\ref{P:MapToCstPfc-1})).
Then
\[
{\wt{\mu}} \circ {\wt{\ph}}_G
  = \mu^{(0)} \circ {\wt{\kp}}_N.
\]
For $g \in G,$
we thus have
\[
({\wt{\mu}} \circ {\wt{\ph}}_G) (u_g)
  = \big( \mu^{(0)} \circ {\wt{\kp}}_N \big) (u_g)
  = w^{(0)}_{\gm_G (g) {\ov{N}}}
  = w_{\gm_G (g)}
  = v_g
  = \pi (u_g).
\]
Since ${\wt{\mu}} \circ\ph_G$ is nondegenerate,
it follows that ${\wt{\mu}} \circ \ph_G = \pi.$
This is the required factorization of~$\pi.$

Now assume~(\ref{P:Factor-Cst}).
Let ${\wt{\pi}}$ be the extension of $\pi$ to $M (C^* (G)),$
and let ${\wt{\ph}}_G$
be as in Proposition~\ref{P:MapToCstPfc}(\ref{P:MapToCstPfc-1}).
Then ${\wt{\mu}} \circ {\wt{\ph}}_G$ also extends~$\pi,$
so ${\wt{\mu}} \circ {\wt{\ph}}_G = {\wt{\pi}}.$
For $g \in G,$ we then have,
using Proposition~\ref{P:MapToCstPfc}(\ref{P:MapToCstPfc-1})
at the third step,
\[
v_g = {\wt{\pi}} (u_g)
    = ({\wt{\mu}} \circ {\wt{\ph}}_G) (u_g)
    = {\wt{\mu}} (u_{\gm_G (g)})
    = w_{\gm_G (g)}.
\]
This proves~(\ref{P:Factor-Gp}).
\end{proof}

\begin{cor}\label{C:KerPhi}
Let $G$ be a locally compact group,
with profinite completion ${\ov{G}},$
and let
$\ph_G \colon C^* (G) \to M \big( C^* \big( {\ov{G}} \big) \big)$
be as in Proposition~\ref{P:MapToCstPfc}.
Let $I \S C^* (G)$
be the intersection
of the kernels of the representations of $C^* (G)$
associated as in Remark~\ref{R:UGMG}
to representations of $G$ with finite range.
Then $\ker (\ph_G) = I.$
\end{cor}

\begin{proof}
Let $a \in \ker (\ph_G).$
Let $v \colon G \to L (H)$ be an arbitrary representation
of $G$ with finite range,
with associated representation $\pi \colon C^* (G) \to L (H).$
Then $v$ factors through~${\ov{G}},$
so Proposition~\ref{P:Factor} implies that
$\pi$ factors through $M \big( C^* \big( {\ov{G}} \big) \big).$
Therefore $\pi (a) = 0.$
This shows that $a \in I.$

Conversely,
let $a \in I.$
For $N \in \CNG,$
let $\kp_N$ and $\sm_N$
be as in the statement of Proposition~\ref{P:MapToCstPfc}.
Then $\kp_N (a) = 0$ for all $N \in \CNG.$
Therefore $\sm_N ( \ph_G (a)) = 0$ for all $N \in \CNG.$
The last sentence of Proposition~\ref{L:IsPCSStr}(\ref{L:IsPCSStr-2}),
applied with ${\ov{G}}$ in place of~$G,$
now implies that $\ph_G (a) = 0.$
\end{proof}

\begin{cor}\label{C:PhiInj}
Let $G$ be a locally compact group,
let
$\ph_G \colon
 C^* (G) \to M \big( C^* \big( {\ov{G}} \big) \big)$
be as in Proposition~\ref{P:MapToCstPfc},
and let $I$ be as in Corollary~\ref{C:KerPhi}.
Then $\ph_G$ is injective \ifo\  $I = 0.$
\end{cor}

\begin{proof}
This is immediate from Corollary~\ref{C:KerPhi}.
\end{proof}

\begin{cor}\label{C:PhiInjImpGRF}
Let $G$ be a locally compact group,
and suppose that the map
$\ph_G \colon
 C^* (G) \to M \big( C^* \big( {\ov{G}} \big) \big)$
of Proposition~\ref{P:MapToCstPfc} is injective.
Then $G$ is \rf.
\end{cor}

\begin{proof}
Every unitary representation of~$G$
with finite range is a direct sum of \fd\  representations.
Corollary~\ref{C:PhiInj}
therefore implies that there is a family
$\big( v^{(\ld)} \big)_{\ld \in \Ld}$ of
\fd\  unitary representations of~$G,$
each with finite range,
such that the representation $\pi \colon C^* (G) \to L (H),$
associated to $v = \bigoplus_{\ld \in \Ld} v^{(\ld)},$
is injective.
Theorem~\ref{T:MultAlg}(\ref{T:MultAlg-Rep})
implies that
its extension ${\wt{\pi}}$ to $M (C^* (G))$
is also injective.
Since ${\wt{\pi}} (u_g) = v_g$ for all $g \in G,$
we conclude that $v$ is injective.
Since $v$ is a direct sum of representations with finite range,
it follows that $G$ is \rf.
\end{proof}

\begin{exa}\label{E:SL3Z}
The converse to Corollary~\ref{C:PhiInjImpGRF}
is false.
Indeed,
the group $G = {\mathrm{SL}}_3 (\Z)$ is \rf,
but its full group \ca\  is not \rfd\  %
(by the main theorem of~\cite{Bk}).
Since $M \big( C^* \big( {\ov{G}} \big) \big)$ is \rfd\  %
(by Proposition~\ref{P:BddIsMult}),
it follows that $\ph_G$ is not injective.
\end{exa}

The method of Example~\ref{E:SL3Z} suggests the following question,
to which we do not know the answer.

\begin{qst}\label{Pb:IfCStarIsRFD}
Is there a \rf\  locally compact group~$G$
such that $C^* (G)$ is \rfd\  %
but $\ph_G$ is not injective?
\end{qst}

The free group $F_2$ on two generators
is a candidate.
The algebra $C^* (F_2)$ is \rfd\  by Theorem~7 of~\cite{Ch}.
Thus, the direct sum of all representations of $F_2$ on
\fd\  Hilbert spaces yields a faithful representation of $C^* (G).$
However,
it is not in general possible to approximate a representation of~$G$
on a \fd\  Hilbert space pointwise
by representations on the same Hilbert space with finite range.
This follows from two facts.
The first is Jordan's Theorem
(see page~91 of~\cite{Jr}),
according to which for every $n \in \N$
there is $l (n) \in \N$ that every finite subgroup of the
unitary group of $M_n$ contains an abelian normal subgroup of index
at most~$l (n).$
(For recent proofs and strengthenings,
see Section~2 of~\cite{BG} and Proposition~2.3 of~\cite{Th}.)
The second is the existence of injective representations of~$F_2$
in the unitary matrix groups $U (n).$
For example, define
\[
u = \left( \begin{matrix}
\sqrt{2} / 2 & 0 & - \sqrt{2} / 2  \\
0            & 1 &             0   \\
\sqrt{2} / 2 & 0 &   \sqrt{2} / 2
\end{matrix} \right)
\andeqn
v = \left( \begin{matrix}
1 &            0 & 0             \\
0 & \sqrt{3} / 2 & - 1 / 2       \\
0 & 1 / 2        & \sqrt{3} / 2
\end{matrix} \right).
\]
Then,
by Section~4 of~\cite{RS},
the elements $(u v)^2$ and $(u v^2)^2$
generate a copy of $F_2$ inside $U (3).$

By itself, the situation described in the previous paragraph
does not show that the
direct sum of all unitary representations of $F_2$ with
finite range
is not faithful on $C^* (F_2).$

We can give two positive results.

\begin{prp}\label{P:RFAb}
Let $G$ be a \rf\  locally compact abelian group.
Then the map
$\ph_G \colon
 C^* (G) \to M \big( C^* \big( {\ov{G}} \big) \big)$
of Proposition~\ref{P:MapToCstPfc} is injective.
\end{prp}

Not all locally compact abelian groups are \rf.
The group $S^1$ with its usual topology,
and $\Q$ with the discrete topology, are counterexamples.
For both of these, $M \big( C^* \big( {\ov{G}} \big) \big) = \C.$

\begin{proof}[Proof of Proposition~\ref{P:RFAb}]
To make the notation less awkward,
set $P = {\ov{G}},$ and abbreviate $\gm_G \colon G \to P$ to~$\gm.$
We let ${\wh{G}}$ and ${\wh{P}}$ denote the Pontryagin duals
of $G$ and~$P,$
and we let ${\wh{\gm}} \colon {\wh{P}} \to {\wh{G}}$
be the map induced by~$\gm.$
We can identify $C^* (G)$
with $C_0 \big( {\wh{G}} \big)$
and $M \big( C^* \big( {\ov{G}} \big) \big)$ with the algebra
$C_{\mathrm{b}} \big( {\wh{P}} \big)$ of bounded
\cfn s on~${\wh{P}}.$
Then $\ph_G$ is given by
$\ph_G (f) = f \circ {\wh{\gm}}.$

Since $G$ is \rf, $\gm$ is injective.
Therefore ${\wh{\gm}}$ has dense range.
(See 24.41(b) of~\cite{HR}.)
So $f \mapsto f \circ {\wh{\gm}}$ is injective.
\end{proof}

\begin{thm}\label{T:RFAmen}
Let $G$ be an amenable \rf\  discrete group.
Then the map
$\ph_G \colon C^* (G) \to M \big( C^* \big( {\ov{G}} \big) \big)$
of Proposition~\ref{P:MapToCstPfc} is injective.
\end{thm}

\begin{proof}
The following argument is taken from the proof of Proposition~3.3
of~\cite{Ws2},
and was pointed out to us by David Kerr.

Since $G$ is amenable,
we can identify $G$ with a subalgebra of $L (l^2 (G))$
via the regular representation.
Let $a \in C^* (G)$ be nonzero.
We show that $a$ is not in the ideal~$I$ of Corollary~\ref{C:KerPhi}.
This will give the result.

\Wolog\  $\| a \| = 1.$
Choose $b \in C^* (G)$ such that $\| b - a \| < \tfrac{1}{4}$
and $b$ is a finite linear combination of the standard unitaries~$u_g.$
That is, there are a finite set $S \subset G$ and numbers $\bt_g \in \C$
for $g \in S$ such that $b = \sum_{g \in S} \bt_g u_g.$
We have $\| b \| > \tfrac{3}{4},$
so there is $\xi \in l^2 (G)$ with finite support such that
$\| \xi \| = 1$ and $\| b \xi \| > \tfrac{1}{2}.$
Let $\dt_g \in l^2 (G)$ be the standard basis element
corresponding to $g \in G,$
and use similar notation for other groups.
Then there are a finite set $T \subset G$ and numbers $\af_g \in \C$
for $g \in T$ such that $\xi = \sum_{g \in T} \af_g \dt_g.$

We let
\[
S T = \big\{ g h \colon {\mbox{$g \in S$ and $h \in T$}} \big\}
\andeqn
T^{-1} = \big\{ g^{-1} \colon g \in T \big\}.
\]
Since $G$ is \rf,
there is $N \in \CNG$ such that the restriction to $S T$ of the
quotient map $G \to G / N$ is injective.
Let $v \colon G \to L (l^2 (G / N))$ be the composition of this
quotient map with the regular \rpn\  of $G / N.$
Let $\pi \colon C^* (G) \to L (l^2 (G / N))$ be the corresponding \hm.
Set $\et = \sum_{g \in T} \af_g \dt_{g N}.$
Then $\| \et \| = 1$ since the vectors $\dt_{g N}$ are orthonormal.
For $g \in S T,$ define
\[
\ld_g = \sum_{h \in S \cap g T^{-1}} \bt_h \af_{h^{-1} g}.
\]
Then
\[
b \xi = \sum_{g \in S T} \ld_g \dt_g
\andeqn
\pi (b) \et = \sum_{g \in S T} \ld_g \dt_{g N}.
\]
As $g$ runs through $S T,$
the elements $\dt_g$ and $\dt_{g N}$ form orthonormal systems
in $l^2 (G)$ and in $l^2 (G / N).$
Therefore
\[
\| \pi (b) \et \|^2 = \sum_{g \in S T} | \ld_g |^2 = \| b \xi \|^2.
\]
So
\[
\| \pi (a) \|
  > \| \pi (b) \| - \tfrac{1}{4}
  \geq \| \pi (b) \et \| - \tfrac{1}{4}
  = \| b \xi \| - \tfrac{1}{4}
  > \tfrac{1}{2} - \tfrac{1}{4}
  = \tfrac{1}{4}.
\]
Therefore $\pi (a) \neq 0.$
Since $\pi$ comes from a \rpn\  of~$G$ which factors through
the finite group $G / N,$
this shows that $a \not\in I.$
\end{proof}

We give an explicit example of a
nonabelian group covered by Theorem~\ref{T:RFAmen},
with a direct proof that the map
$\ph_G \colon
   C^* (G) \to M \big( C^* \big( {\ov{G}} \big) \big)$
of Proposition~\ref{P:MapToCstPfc} is injective.

\begin{exa}\label{E:DiscHsbg}
Let $G$ be the discrete Heisenberg group,
that is,
\[
G = \left\{ \left( \begin{matrix}
  1     &  n     &  l        \\
  0     &  1     &  m        \\
  0     &  0     &  1
\end{matrix} \right) \colon
l, m, n \in \Z \right\}.
\]
We can identify $G$ with the group generated by elements
$g,$ $h,$ and~$z,$
subject to the relations
\[
g h = z h g,
\,\,\,\,\,\,
z g = g z,
\andeqn
z h = h z.
\]
Here,
\[
g = \left( \begin{matrix}
  1     &  0     &  0        \\
  0     &  1     &  1        \\
  0     &  0     &  1
\end{matrix} \right),
\,\,\,\,\,\,
h = \left( \begin{matrix}
  1     &  1     &  0        \\
  0     &  1     &  0        \\
  0     &  0     &  1
\end{matrix} \right),
\andeqn
z = \left( \begin{matrix}
  1     &  0     &  1        \\
  0     &  1     &  0        \\
  0     &  0     &  1
\end{matrix} \right).
\]
(See Section VII.5 of~\cite{Dv},
where these elements are called $u,$ $v,$ and~$w.$)

By Theorem VII.5.1 of~\cite{Dv}, a complete set of representatives
of the unitary equivalence classes of irreducible \rpn s of $C^* (G)$
is given by the \rpn s $\pi^{(n, k, \af, \bt)}$ defined as follows.
We require that $n \in \N,$ that $k \in \{ 1, 2, \ldots, n \}$
and be relatively prime to~$n,$
and that $\af$ and $\bt$ both be in the arc
\[
J_n = \big\{ \exp (2 \pi i t) \colon 0 \leq t < \tfrac{1}{n} \big\}.
\]
Let $s_n \in M_n$ be the cyclic shift,
defined on the standard basis vectors $\dt_j \in \C^n$
by $s_n \dt_n = \dt_1$ and $s_n \dt_j = \dt_{j + 1}$
for $j = 1, 2, \ldots, n.$
Then $\pi^{(n, k, \af, \bt)}$ is determined by
\[
\pi^{(n, k, \af, \bt)} (g)
 = \af \cdot \diag \big( 1, \, e^{2 \pi i k / n},
          \, e^{2 \pi i \cdot 2 k / n},
          \, \ldots, e^{2 \pi i (n - 1) k / n} \big),
\]
\[
\pi^{(n, k, \af, \bt)} (h)
 = \bt s_n,
\andeqn
\pi^{(n, k, \af, \bt)} (z) = e^{2 \pi i k / n} \cdot 1.
\]
(It is not stated in~\cite{Dv},
but one easily checks that these representations actually all exist.)

Remark VII.5.2 of~\cite{Dv} implies that the intersection
of the kernels of these representations is~$\{ 0 \}.$
(This shows that $C^* (G)$ is \rfd.)
To prove injectivity of $\ph_G$ using
Corollary~\ref{C:KerPhi},
it therefore suffices to find, for each $n$ and~$k,$
a dense subset $R_{n, k}$
of the collection of allowed values of $\af$ and $\bt$ such that,
for $\af, \bt \in R_{n, k},$
the representation $w^{(n, k, \af, \bt)}$
of~$G$ has finite range.
One can check that the set
$R_{n, k}$ of roots of unity in~$J_n$
meets this requirement.
\end{exa}

Let $G$ be a locally compact group.
Then we have two \pcas s on $C^* (G).$
One is the \pfpcas.
The other is the \pcas\  gotten
by applying the map
$\ph_G \colon
 C^* (G) \to M \big( C^* \big( {\ov{G}} \big) \big)$
of Proposition~\ref{P:MapToCstPfc},
and then using the \csn s of the \pcas\  %
on $M \big( C^* \big( {\ov{G}} \big) \big)$
from Definition~\ref{D:ProCStPfGp}.
We expect that these two \pcas s are usually inequivalent,
and in general do not even yield the same completion
in any reasonable sense,
even when $\ph_G$ is injective.
See Example~\ref{E:NotSameComp} below.

These pro-C*-algebra structures
are at least comparable.
This is conveniently expressed as follows.

\begin{prp}\label{P:CtInProC}
Let $G$ be a locally compact group.
Give $M \big( C^* \big( {\ov{G}} \big) \big)$ the \pcas\  of
Proposition~\ref{L:IsPCSStr}(\ref{L:IsPCSStr-1}).
Give $C^* (G)$ its \pfpcas\  (\Def{D:Pfpca}).
Then the map
$\ph_G \colon
 C^* (G) \to M \big( C^* \big( {\ov{G}} \big) \big)$
of Proposition~\ref{P:MapToCstPfc} is continuous
for the topologies induced by these \pcas s.
Moreover, $\ph_G$ extends to a \ct\  \hm\  %
\[
{\ov{\ph}}_G \colon
  {\ov{C^* (G)}} \to
   {\ov{M \big( C^* \big( {\ov{G}} \big) \big)}}.
\]
\end{prp}

\begin{proof}
The second statement follows from the first.

To prove the first statement,
it is enough to prove the following.
Let $p$ be any \csn\  in the \pcas\  %
on $M \big( C^* \big( {\ov{G}} \big) \big).$
Then the \csn\   $p \circ \ph_G$ is in the
\pcas\  on $C^* (G).$

So let $p$ be such a \csn.
By definition,
there is $N \in \CNGb,$
with corresponding \hm\  %
$\kp_N \colon
  C^* \big( {\ov{G}} \big) \to C^* \big( {\ov{G}} / N \big)$
and extension
${\wt{\kp}}_N \colon
  M \big( C^* \big( {\ov{G}} \big) \big)
   \to C^* \big( {\ov{G}} / N \big),$
such that $p (a) = \| {\wt{\kp}}_N (a) \|$
for all $a \in M \big( C^* \big( {\ov{G}} \big) \big).$
Then $C^* (G) / \ker ( p \circ \ph_G )$
is isomorphic to
a subalgebra of $C^* \big( {\ov{G}} / N \big),$
and is hence \fd.
Therefore $p \circ \ph_G$ is in the
\pcas\  on $C^* (G).$
\end{proof}

\begin{thm}\label{T-IdPhiBar}
Let $G$ be a locally compact group.
Let $R$ be a set consisting of one representative
$v \colon G \to L (H_v)$
of each unitary equivalence class of \fd\  irreducible
representations of~$G.$
Let
\[
F = \big\{ v \in R \colon {\mbox{$v$ has finite range}} \big\}.
\]
Then there is a commutative diagram
\[
\xymatrix{
{\ov{C^* (G)}} \ar[r]^{{\ov{\ph}}_G} \ar[d]_{\af}
  & {\ov{M \big( C^* \big( {\ov{G}} \big) \big)}} \ar[d]_{\bt}   \\
\prod_{v \in R} L (H_{v}) \ar[r]_{\rh} & \prod_{v \in F} L (H_{v}),
}
\]
in which the vertical maps $\af$ and $\bt$ are isomorphisms of
topological algebras,
the top horizontal map is as in Proposition~\ref{P:CtInProC},
and the bottom horizontal map $\rh$ is
the obvious projection map.
\end{thm}

\begin{proof}
For $v \in R$ let $\pi^v \colon C^* (G) \to L (H_v)$
be the corresponding representation of $C^* (G).$
Let $S$ be the set of $v \in R$ which extend
continuously to representations of~${\ov{G}}.$

The map $\af$ is obtained from Proposition~\ref{P-IdPfCstComp},
using the set of representations $\{ \pi^v \colon v \in R \}.$
To obtain~$\bt,$
we begin by applying
Proposition~\ref{L:IsPCSStr}(\ref{L:IsPCSStr-2})
and parts (\ref{P:1-19A-2}) and~(\ref{P:1-19A-3})
of Proposition~\ref{P:1-19A}
with $A = C^* \big( {\ov{G}} \big).$
This identifies $M \big( C^* \big( {\ov{G}} \big) \big)$
and ${\ov{M \big( C^* \big( {\ov{G}} \big) \big)}}$
as
\[
M \big( C^* \big( {\ov{G}} \big) \big)
  = \left\{ b \in \prod_{v \in S} L (H_v) \colon
          \sup_{v \in S} \| b_{s} \| < \I \right\}
\andeqn
{\ov{M \big( C^* \big( {\ov{G}} \big) \big)}}
     = \prod_{v \in S} L (H_v),
\]
and we take $\bt$ to be
induced by the obvious inclusion.

We next identify~$\rh.$
Let $a \in C^* (G),$
and let ${\ov{a}}$ be its image in ${\ov{C^* (G)}}.$
Then $\af ({\ov{a}}) = ( \pi^v (a) )_{v \in R}$
by Proposition~\ref{P-IdPfCstComp}.
Using Proposition \ref{P:MapToCstPfc}(\ref{P:MapToCstPfc-1}),
we get
$(\bt \circ {\ov{\ph}}_G) ({\ov{a}}) = ( \pi^v (a) )_{v \in S}.$
It is now obvious that the diagram commutes.

It remains to identify $S$ with~$F.$
Since ${\ov{G}}$ is profinite,
this identification is immediate from Lemma~\ref{L:FinRange}.
\end{proof}

\begin{exa}\label{E:NotSameComp}
Take $G = \Z.$
Referring to Theorem~\ref{T-IdPhiBar},
we can identify $R$ with the dual group ${\wh{G}},$
that is, with the unit circle~$S^1.$
We thus get the identification $C^* (G) = C (S^1).$
We can identify $F$ as
\[
F = \big\{ \zt \in S^1 \colon
    {\mbox{there is $n \in \N$ such that $\zt^n = 1$}} \big\}.
\]
The topology from the \pfpcas\  on $C (S^1)$
is the topology of pointwise convergence on~$S^1,$
and, by Theorem~\ref{T-IdPhiBar}, the completion
is the set of all functions from $S^1$ to~$\C.$
(See Example~\ref{E:CXPfCmp}.)
The topology from the \pcas\  on $C (S^1) = C^* (G)$
obtained via~$\ph_G$
is the topology of pointwise convergence on~$F.$
It is still faithful,
but its completion is the set of all functions from $F$ to~$\C.$
\end{exa}

\begin{cor}\label{P:phBarSj}
Let $G$ be a locally compact group.
Then the map
${\ov{\ph}}_G \colon
  {\ov{C^* (G)}} \to
   {\ov{M \big( C^* \big( {\ov{G}} \big) \big)}}$
of Proposition~\ref{P:CtInProC} is surjective.
\end{cor}

\begin{proof}
The map~$\rh$ of Theorem~\ref{T-IdPhiBar} is always surjective.
\end{proof}

\begin{cor}\label{C-phBarIso}
Let $G$ be a locally compact group.
Then the following are equivalent.
\begin{enumerate}
\item\label{C-phBarIso-1}
The map
${\ov{\ph}}_G \colon
  {\ov{C^* (G)}} \to
   {\ov{M \big( C^* \big( {\ov{G}} \big) \big)}}$
of Proposition~\ref{P:CtInProC} is bijective.
\item\label{C-phBarIso-2}
Every finite dimensional irreducible representation of~$G$
has finite range.
\item\label{C-phBarIso-3}
Every finite dimensional representation of~$G$
has finite range.
\end{enumerate}
Moreover, when these conditions are satisfied,
${\ov{\ph}}_G$ is a \hme.
\end{cor}

\begin{proof}
The equivalence of (\ref{C-phBarIso-1}) and~(\ref{C-phBarIso-2})
is clear from Theorem~\ref{T-IdPhiBar}.
The equivalence of (\ref{C-phBarIso-2}) and~(\ref{C-phBarIso-3})
follows from
the fact that every finite dimensional representation of~$G$
is a finite direct sum of
finite dimensional irreducible representations.
\end{proof}

It is tempting to hope that
the map ${\ov{\ph}}_G$
of Proposition~\ref{P:CtInProC} is a \hme\
\ifo\  $G$ is profinite.
This is true when $G$ is compact,
but fails in general.

\begin{prp}\label{P:CompProf}
Let $G$ be a compact group.
Then the map
${\ov{\ph}}_G \colon
  {\ov{C^* (G)}} \to
   {\ov{M \big( C^* \big( {\ov{G}} \big) \big)}}$
of Proposition~\ref{P:CtInProC} is a \hme\
\ifo\  $G$ is profinite.
\end{prp}

\begin{proof}
By Corollary~\ref{C-phBarIso},
we must check that $G$ is profinite
\ifo\  every finite dimensional representation of~$G$
has finite range.

If $G$ is profinite,
then every finite dimensional representation of~$G$
has finite range by Lemma~\ref{L:FinRange}.

Suppose $G$ is compact and not profinite.
Let ${\ov{G}}$ be the profinite completion of
$G$ and let $\gm_G \colon G \to {\ov{G}}$
the canonical map, as in Notation~\ref{N:NBar}.
Then $\gm_G (G)$ is dense in ${\ov{G}}$ by construction,
and its range is compact,
so $\gm_G$ is surjective.
Let $H = \ker (\gm_G).$
Then $H$ is compact and nontrivial,
and hence has a
nontrivial finite dimensional irreducible representation~$\sigma_0.$
The range of $\sigma$ is infinite because
$H$ contains no nontrivial subgroups with finite index.
Induce $\sigma_0$ to a representation $\sigma$ of~$G.$
The restriction of $\sigma$ to $H$ contains $\sigma_0$
as a summand (by Frobenius reciprocity;
see for example Theorem 7.4.1 in~\cite{DE}).
Therefore, when $\sigma$ is written as a
direct sum of finite dimensional irreducible representations
of~$G,$
at least one of them, say $\pi,$ has the property
that $\sigma_0$ is a summand in $\pi |_H.$
This representation $\pi$ is a finite dimensional representation
of $G$ whose range is not finite.
\end{proof}

\begin{exa}\label{E-DirSumZ2}
Let $G$ be the abelian group
$G = \bigoplus_{n = 1}^{\I} \Z / 2 \Z,$
with the discrete topology.
Then $G$ is not profinite because it is not compact.
(In fact,
its profinite completion is $\prod_{n = 1}^{\I} \Z / 2 \Z.$)
Nevertheless, the map
${\ov{\ph}}_G \colon
  {\ov{C^* (G)}} \to
   {\ov{M \big( C^* \big( {\ov{G}} \big) \big)}}$
of Proposition~\ref{P:CtInProC} is a \hme.

We verify the criterion of
Corollary~\ref{C-phBarIso}(\ref{C-phBarIso-2}).
The irreducible representations of~$G$
are all one dimensional.
Since every element of $G$ has order $1$ or~$2,$
and $S^1$ has only two such elements,
the range of every irreducible representation of~$G$
has at most two elements.
\end{exa}

We give a second example,
an infinite countable amenable group for which
both ${\ov{C^* (G)}}$
and ${\ov{M \big( C^* \big( {\ov{G}} \big) \big)}}$
are just~$\C.$

\begin{exa}\label{E-InfAltGp}
Let $G$ be the group of finitely
supported even permutations of a countable infinite set,
with the discrete topology.
This group is countable and simple.
It is amenable,
since it is the increasing union of finite subgroups.
It is not profinite because it is not compact.
Nevertheless, we claim that the map
${\ov{\ph}}_G \colon
  {\ov{C^* (G)}} \to
   {\ov{M \big( C^* \big( {\ov{G}} \big) \big)}}$
of Proposition~\ref{P:CtInProC} is a \hme.

We verify the criterion of
Corollary~\ref{C-phBarIso}(\ref{C-phBarIso-3}),
by showing that $G$
has no nontrivial finite dimensional representations.
Since $G$ is simple,
it suffices to show that $G$
has no faithful finite dimensional representations.

Suppose $u \colon G \to L (\C^n)$ is a faithful representation.
Jordan's Theorem
(see page~91 of~\cite{Jr}, or Theorem~2.1 of~\cite{BG})
provides $l \in \N$ such that every finite subgroup
of $G$ contains an abelian normal subgroup of index
at most~$l.$
Since $G$ contains a copy of every finite group,
this is obviously impossible.
\end{exa}

\end{document}